\documentclass[american]{article}
\usepackage[T1]{fontenc}
\usepackage[latin9]{inputenc}
\usepackage{amsmath}
\usepackage{amssymb}
\usepackage{amsthm}
\usepackage{epsfig}
\usepackage{amssymb}
\usepackage{setspace}
\usepackage{fullpage}

\usepackage{babel}

\newtheorem{Theorem}{Theorem}
\newtheorem{Lemma}{Lemma}
\newtheorem{proposition}{Proposition}

\begin{document}

\title{Generalized Stable Matching in Bipartite Networks}

\author{Ankur Mani, Asuman Ozdaglar, Alex (Sandy) Pentland\\
amani@mit.edu, asuman@mit.edu, pentland@mit.edu}

\date{}

\maketitle

\large

\begin{abstract}

In this paper we study the generalized version of weighted matching in bipartite networks. Consider a weighted matching in a bipartite network in which the nodes derive value from the split of the matching edge assigned to them if they are matched. The value a node derives from the split depends both on the split as well as the partner the node is matched to. We assume that the value of a split to the node is continuous and strictly increasing in the part of the split assigned to the node. A stable weighted matching is a matching and splits on the edges in the matching such that no two adjacent nodes in the network can split the edge between them so that both of them can derive a higher value than in the matching.  We extend the weighted matching problem to this general case and study the existence of a stable weighted matching. We also present an algorithm that converges to a stable weighted matching. The algorithm generalizes the Hungarian algorithm for bipartite matching. Faster algorithms can be made when there is more structure on the value functions.

\end{abstract}

\doublespacing


\section{Introduction}

In this paper we analyze the following problem. Consider a weighted matching in a bipartite network in which the nodes derive value from the split of the matching edge assigned to them if they are matched. The value a node derives from the split depends both on the split as well as the partner the node is matched to. We assume that the value of a split to the node is continuous and strictly increasing in the part of the split assigned to the node. A stable weighted matching is a matching and splits on the edges in the matching such that no two adjacent nodes in the network can split the edge between them so that both of them can derive a higher value than in the matching.  We extend the weighted matching problem to this general case and study the existence of a stable weighted matching. We also present an algorithm that converges to a stable weighted matching. The algorithm generalizes the Hungarian algorithm \cite{Hungarian-Method} for bipartite matching. Faster algorithms can be made when there is more structure on the value functions.

Weighted matching in bipartite networks has been studied in the context of linear valuations \cite{Combinatorial-Optimization}. The problem is often posed as such.

In a bipartite network $S=\left(A\cup B, E \subseteq A \times B\right)$, whose nodes belong to $A \cup B$ and whose edges connect nodes from $A$ to nodes in $B$ with weights $w\left(i, j\right)$ for the edge $\left(i,j\right)$ that can be split between $i$ and $j$ as $s_i$ and $s_j$ to give them values $V_i=s_i$ and $V_j=s_j$, find a matching $M^*$ with characteristic function $\chi^M$ that maximizes the sum of weights of edges in the matching
\begin{align*}
&{\sum_{\left(i,j\right) \in E}}{w\left(i,j\right)\chi\left(i,j\right)}
\end{align*}
The characteristic function $\chi^M$ must satisfy the following constraints to be the characteristic function of a matching.
\begin{align}
&\sum_{\left(i,j\right) \in E}\chi\left(i,j\right)\leq1,\forall i\in A \cup B\\
&\chi\left(i,j\right)\geq0,\forall\left(i,j\right)\in E\\
&\chi\left(i,j\right) \in \{0,1\},\forall\left(i,j\right)\in E
\end{align}
The last constraint is an integer constraint and can be neglected since the corners of the polytope resulting from the remaining constraints are integral. The above is called the maximum weight matching problem. In a finite graph with finite weights, the optimal solution exists and the optimal value is finite. The stable matching problem is the dual of the maximum weight matching problem which is to find the minimum sum of values given to the nodes in the network
\begin{align}
&\min\sum_{i \in A \cup B}V_{i}\\
&\mbox{such that } V_{i}+V_{j} \geq w\left(i,j\right),\forall \left(i,j\right) \in E\\
&V_{i} \geq 0
\end{align}
The existence of a stable matching is evident from the finiteness of the optimal value in the maximum weight matching problem. In fact at optimal solution for the stable matching problem, for any edge $\left(i,j\right) \in M^*$, $V_{i}+V_{j} = w\left(i,j\right)$. For any maximum weighted matching, there exists splits $\bold{s} = \bold{V}$ that is an optimal solution to the stable matching problem. This problem has been well understood and several algorithms have been proposed to find the optimal stable matching. The extensions, when the values $V$ are increasing functions of the split and do not depend upon the edge, can be reduced to the above problem.

We study the problem when the values $V$ depend upon the edge as well as the part of the split given to the node. The stable matching problem in this case is as such. Find a matching $M^\circ$ and a split $\bold{s}$ such that
\begin{align}
&s_i + s_j = w\left(i,j\right), \forall \left(i,j\right) \in M^\circ\\
&v_i = V_i\left(j, s_i\right), v_j = V_j\left(i, s_j\right), \forall \left(i,j\right) \in M^\circ\\
&V^{-1}_{i}\left(j,v_i\right) + V^{-1}_{j}\left(i,v_j\right) \geq w\left(i,j\right),\forall \left(i,j\right) \in E\\
&s_{i} \geq 0, \forall i \in A \cup B
\end{align}
The existence of such a matching and a split is not evident. In this paper, we show that such a matching and a split exists and we give an algorithm to find such a matching and a split. The problem features in many practical problems. We give a few examples.

Consider the stable marriage problem and related problems studied in \cite{Gale-Shapley} and later by several others. A survey of related literature can be found in \cite{Roth-Sotomayor}. The classical formulation assumes exogeneous partner preferences. Other formulations including \cite{Theory-of-Marriage} study endogeneous partner preferences arising from types of partners. An important and more realistic formulation is to consider that utilities of individuals in a marriage depends both on the type or the identity of the partner as well as the effort the partner puts in the marriage. In this scenario, a stable marriage is the one in which the neither partner in the marriage has a proposal for mariage in which the partner will have a higher utility.

Another example is the exchanges in buyer-seller networks \cite{Kranton-Minehart}. The problem has been studied in the context of indivisible goods. An important scenario is the case of divisible goods with the buyer-seller relations being exclusive. When the preferences for the goods are strictly convex, continuous and strongly monotone, we observe a connected contract curve or the set of individually rational pareto-efficient exchanges between any adjacent buyer-seller pair in the network. As we move along the contract curve in a given direction, the utility of buyer/seller strictly decreases and the utility of seller/buyer strictly increases. The stable set of exchanges in this network is the one in which all exchanges are stable or no adjacent buyer-seller pair in the network can do better by simultaneously breaking their current contracts and forming a new contract among them.

An important example is the study of bargaining in networks. This problem has recently been studied widely and takes the form of the stable matching problem in teh case of linear utilities. However, often in real life bargaining situations, the utility is non-trasferable between the bargaining parties through a quasilinear numeraire. In such situations as the sum of offers to the two parties in bargaining is not constant. A stable bargaining solution in this case takes a different form as studied in this paper. Another line of work that can benefit by the results in this paper is the work on social games introduced in \cite{Social-Games}.

The organization of the rest of the paper is as such. In the next section, we introduce the setup. We try to maintain the notations close to the notations in the matching literature while introduce additional terminology as required. In section 3, we show introduce some important concepts that are needed to prove the existence of a stable matching. Finally in section 4, we show a contructive proof and an algorithm to find the stable matching.


\section{Setup}

In this section we formulate the problem and introduce necessary terminology.

\subsection{Network and Payoffs}

Assume $A$ and $B$ are two finite and mutually exclusive sets of nodes and $X=A \cup B$. A bipartite network between $A$ and $B$ is a graph $S=\left(X, E \subseteq A \times B\right)$, whose nodes belong to $X$ and whose edges connect nodes from $A$ to nodes in $B$. Given a bipartite network $S$, we will refer to the set of nodes as $X^S$, the sets of nodes in $A$ and $B$ as $A^S$ and $B^S$ respectively, and the set of edges as $E^S$ when necessary. When the node set $X$ is understood, we will refer to the network by the edge set $E$. Without loss of generality, we will assume that the graph $S$ is connected.

A set of nodes $X' \subseteq X$ induces a {\bf subgraph} $S_{|X'}=\left(X', E_{|X'}\right)$ of $S$, such that $E_{|X'} = \{\left(i,j\right) \in E : i \in X', j \in X'\}$.

The set of neighbors of a node $i \in A$ is $Nbr^S\left(i\right) = \{j \in B : \left(i,j\right) \in E\}$. The set of neighbors of a node of $j \in B$ is $Nbr^S\left(j\right) = \{i \in A : \left(i,j\right) \in E\}$. When the context is well understood, we will also refer to $Nbr^S\left(i\right)$ as $Nbr^E\left(i\right)$ or just $Nbr\left(i\right)$. The set of neighboring nodes of $x \subseteq X$ is $Nbr^S\left(x\right) = \cup_{i \in x}Nbr^S\left(i\right)$.

A {\bf weight} function $w:E \rightarrow \mathbb{R}_{+}$ assigns a weight to each edge. An edge $\left(i, j\right)$ with $i \in A$ and $j \in B$ has a weight $w\left(i, j\right)$ that can be split between $i$ and $j$.

A {\bf split }$s^{\left(i,j\right)}$ on the edge $\left(i,j\right)$ is the pair $\left(s_{i},s_{j}\right)$, with $s_{i}+s_{j}=w\left(i,j\right)$.

The nodes derive payoffs from the part of the split given to them. For a split $s^{\left(i,j\right)}$, the payoff of the node $i$ is $u_{i}\left(j, s_{i}\right)$ and the payoff of the node $j$ is $u_{j}\left(i, s_{j}\right)$. The payoff of a node depends upon both the part of the split given to the node and the edge on which the split is made. Thus for each edge the payoff of a node is a unique function of the part of the split given to the person. We assume that these payoff functions are strictly increasing and continuous and hence they are invertible and the inverse functions are also strictly increasing and continuous.

We also define payoffs of nodes as a function of its neighbor when the split is made between them. The payoff of node $j$ for the split on edge $\left(i,j\right)$ is a function of the payoff of node $i$ is $v_{j}\left(i,v_{i}\right) = u_{j}\left(i,w\left(i,j\right)-u_{i}^{-1}\left(j,v_{i}\right)\right)$. So the payoffs of neighboring nodes for the split between them are strictly decreasing and continuous with respect to one another. For simplicity, we will assume that the function $v_i\left(j, x\right)$ is defined for all $x \in \bold{R}$ and is strictly decreasing and continuous. We will refer to these functions as {\bf pareto payoff function}.

\subsection{Matching}

A {\bf matching} in the bipartite network is a subset of edges $M \subseteq E$ such that no two edges in $M$ share a common node. The size of the matching $|M|$ is the number of edges in the matching. The matching defines a characteristic function on the set of edges in the bipartite network $\chi^{M}:E\rightarrow\{0,1\}$ where $\forall \left(i,j\right) \in E$,

$\chi^{M}\left(i,j\right) =
\begin{cases}
	1 & \mbox{if} \left(i,j\right)\in M\\
	0 & \mbox{if} \left(i,j\right)\notin M
\end{cases}$

such that\\
$\forall i \in A$, $\sum_{j \in Nbr\left(i\right)} \chi^{M}\left(i,j\right)\leq1$ and\\
$\forall j \in B$, $\sum_{i \in Nbr\left(j\right)} \chi^{M}\left(i,j\right)\leq1$.

The {\bf match} for a node $i \in A$ in the matching $M$ is\\
$M\left(i\right) =
\begin{cases}
j \in Nbr\left(i\right) & \mbox{if} \chi^{M}\left(i,j\right) = 1\\
null & \mbox{if} \sum_{j \in Nbr\left(i\right)} \chi^{M}\left(i,j\right) = 0
\end{cases}$.

We say that a node $i$ is {\bf matched} if $M\left(i\right) \neq null$.

A split for a matching $M$ is a function $s^{M} : A \cup B \rightarrow \mathbb{R}_{+}$ such that\\
$\forall \left(i,j\right) \in M, s^{M}\left(i\right) + s^{M}\left(j\right) = w\left(i,j\right)\\
\forall i \in A, s^{M}\left(i\right) = 0 \mbox{if } \sum_{j \in Nbr\left(i\right)} \chi^{M}\left(i,j\right) = 0\\
\forall j \in B, s^{M}\left(j\right) = 0 \mbox{if } \sum_{i \in Nbr\left(j\right)} \chi^{M}\left(i,j\right) = 0$.

A {\bf weighted matching} $\left(M, s^{M}\right)$ is a pair matching and a split for the matching.

The {\bf payoff profile} $\bold{U}\left(M,s^{M}\right)$ for a weighted matching $\left(M, s^{M}\right)$ is a vector where each element is the payoff of a node for the given weighted matching. The payoff of a node $i$ for the weighted matching $\left(M, s^{M}\right)$ is\\
$U_{i}\left(M,s^{M}\right) = u_{i}\left(M\left(i\right),s^{M}\left(i\right)\right)$.

A weighted matching is {\bf stable} if $\forall \left(i,j\right) \in E$ and all splits $s^{\left(i,j\right)}$ on $\left(i,j\right)$,\\
$u_{i}\left(j,s_{i}\right) \leq U_{i}\left(M,s^{M}\right)$ and
$u_{j}\left(i,s_{j}\right) \leq U_{j}\left(M,s^{M}\right)$.

\subsection{Paths}

A {\bf path} in the network $S$ is a subgraph $P = \left(X^{P}, E^{P}\right)$, where $X^P \subseteq A \cup B$ is a set of nodes and $E^P$ is a set of edges with both end points in $X^P$ such that two nodes in $X^P$ have exactly one edge in $E^P$ and all other nodes in $X^P$ have exactly two edges in $E^P$. The two nodes with exactly one edge will be referred to as the {\bf end nodes}. A path also induces index function over its nodes as follows:

\begin{enumerate}

\item Pick an end node and set its index as $0$. This node is the {\bf source node}.
\item Set $i \leftarrow \mbox{last indexed node}$.
\item If $i$ is an end node, then stop else index the only unindexed neighbor of $i$ as (index of $i$)$+1$. The end node with the highest index is the {\bf sink node}.
\end{enumerate}
This index generates a sequence of nodes $\{x_{n}^{P}\}$ where the subscript stands for the index and $x_{0}^{P}$ and $x_{|X|}^{P}$ are end nodes.

Thus a path $P$ of length $N$ can be seen as a sequence $\{x_{n}\}_{n \in \{0,...,N\}}$ of nodes in $A \cup B$, such that $\forall n<N, \left(x_{n},x_{n+1}\right) \in E$. We call this a path from the source to the sink node. Alternatively, a path is a sequence of nodes such that for each node in the sequence shares edges with both the immediately preceding and immediately succeeding nodes. When the source and sink is determined for the path $P$, we will refer a path from the source $i$ to sink $j$ as $P_{i,j}$. The reverse path from $j$ to $i$ will be refered to as $P_{j,i}$.

A {\bf subpath} $P'$ of a path $P$ is a connected subgraph of the path $P$. Alternatively, a subpath $P' \subseteq P$ between nodes $x_{m}, x_{M} \in X^P$ is a subsequence $\{x_{n}^P\}_{n \in \{m,...,M\}}$. A subpath is a path by itself.

The union of two paths $P_1$ and $P_2$ is also a path $P_3=P_1 \cup P_2$ if $P_1$ and $P_2$ share exactly one end node.

We will denote the set of paths from $i$ to $j$ in a network $S$ as $\bold{P}^S_{i,j}$.


\subsection{Offers}
	
An {\bf offer profile} is a vector $\bold{O} \in \bold{R}^{A \cup B}$ where the element $O_{i}$ is node $i$'s offer. We will denote the restriction of an offer profile $\bold{O}$ to a set of nodes $X$ as $\bold{O}_{|X}$.

An offer profile $\bold{O}$ is {\bf feasible}, if $\exists \left(M, s^{M}\right)$ weighted matching, with payoff profile $\bold{U}\left(M,s^{M}\right)=\bold{O}$.

An offer profile $\bold{O}$ is {\bf stable} if $\forall (i,j) \in E$, $O_j \geq v_{j}\left(i, O_{i}\right)$.

Using the definition and properties of the pareto payoff functions, we can reformulate the stable matching problem as such. Find a matching $M^\circ$ and a split $\bold{s}$ such that
\begin{align}
&s_i + s_j = w\left(i,j\right), \forall \left(i,j\right) \in M^\circ\\
&o_i = u_i\left(j, s_i\right), o_j = u_j\left(i, s_j\right), \forall \left(i,j\right) \in M^\circ\\
&o_j \geq v_{j}\left(i, o_{i}\right)\},\forall \left(i,j\right) \in E\\
&s_{i} \geq 0, \forall i \in A \cup B
\end{align}

The first, second and fourth inequalities provide the constraints for the offer profile to be feasible and the third inequality provide the constraints for the offer profile to be stable. Thus, if we have a feasible and stable offer profile, then we have a weighted stable matching. Hence, in this paper, we will focus on finding a feasible and stable offer profile.

Given an offer profile $\bold{O}$, the {\bf equality subgraph} $EQ\left(\bold{O}\right)$ is the subset of edges $E$ with $EQ\left(\bold{O}\right)=\{\left(i,j\right) \in E : O_{j}=v_{j}\left(i, O_{i}\right)\}$.

We will refer to the neighbors of a node $i$ in the equality subgraph $EQ\left(\bold{O}\right)$ as $Nbr^{EQ\left(\bold{O}\right)}\left(i\right)$.

Given an offer profile $\bold{O}$, a path $P$ is feasible if $E^P \subseteq EQ\left(\bold{O}\right)$.


Given a node $i$ with offer $o_{i} \geq 0$, a path $P$ with an end node $i$ induces an offer for each node $x_{n}^{P}$ in the path as follows:
\begin{itemize}
\item $x_{0}^{P}=i$, Therefore, $O_{x_{0}^{P}}=o_{i}$
\item $\forall 0<n<N, O_{x_{n}^{P}}=v_{x_{n}^{P}}\left(x_{n-1}^{P},O_{x_{n-1}^{P}}\right)$.
\end{itemize}

For any pair of nodes $i,j$ and a path $P$ from $i$ to $j$, we define the {\bf path induced offer function} $f^P_{i,j}:\bold{R} \rightarrow \bold{R}$ where $f^P_{i,j}\left(x\right)$ is the offer that $P$ induces for $j$ given $i$ has the offer $x$. Clearly, $f^P_{i,j}$ is continuous since the pareto payoff functions are continuous. Also $f^P_{i,j}$ is strictly increasing if both $i,j \in A$ or both $i,j \in B$ and strictly decreasing if either $i \in A$ and $j \in B$ or $i \in B$ and $j \in A$ since pareto payoff functions are strictly decreasing.

Given a node $i$ with offer $x$ and another node $i'$, a path $P^*_{i,i'}$ from $i$ to $i'$ is {\bf maximum offer inducing path} from $i$ to $i'$ given the offer $x$ on $i$ if $$P^*_{i,i'} \in \arg \max_{P \in \bold{P}^S_{i,i'}}{f^P_{i,i'}\left(x\right)}.$$ The maximum offer inducing paths and the maximum path induced offers have important properties that we will use for the main result. In the following two lemmas we state these  properties.

\begin{Lemma} \label{lem:max-path}

Assume $i,i' \in A$ and $x \in \bold{R}$. Assume $P^*_{i,i'}$ is a maximum offer inducing path from $i$ to $i'$ given the offer $x$ on $i$. If $P'_{i,i''} \subseteq P^*_{i,i'}$ is the subpath between nodes $i$ and $i'' \in A^{P^*_{i,i'}}$. Then $P'_{i,i''}$ is a maximum offer inducing path from $i$ to $i''$ given the offer $x$ on $i$.

\end{Lemma}

\begin{proof}

The proof follows from the principle of optimality \cite{Dynamic-Programming} and is omitted.


\end{proof}

\begin{Lemma} \label{lem:conn-stbl}

Pick $i \in A$ and $o_i \in \bold{R}$.\\
For all $i' \in A$ set $$O_{i'} =  \max_{P \in \bold{P^S_{i,i'}}}{f^P_{i,i'}\left(o_i\right)}.$$\\
For all $j \in B$, set $$O_{j} = \max_{i' \in Nbr\left(j\right)}{v_j\left(i',O_{i'}\right)}.$$

Then the following hold true about the equality subgraph $EQ\left(\bold{O}\right)$

\begin{enumerate}

\item the equality subgraph $EQ\left(\bold{O}\right)$ is connected and the offer profile $\bold{O}$ is stable. Therefore all nodes $j \in B$ have at least one edge in the equality subgraph.

\item $\forall \left(i',j\right) \in E$, either $\left(i',j\right) \in EQ\left(\bold{O}\right)$ or all paths from $i$ to $i'$ in $EQ\left(\bold{O}\right)$ include at least one of the nodes $i'' \in Nbr^{EQ\left(\bold{O}\right)}\left(j\right)$.


\item If $\left(i',j'\right), \left(i',j''\right) \in E \setminus EQ\left(\bold{O}\right)$ and $v_{i'}\left(j',O_{j'}\right) > v_{i'}\left(j'',O_{j''}\right)$, then all paths from $i$ to $j'$ in $EQ\left(\bold{O}\right)$ include at least one node in $Nbr^{EQ\left(\bold{O}\right)}\left(j''\right)$.

\item If $\left(i',j'\right) \in E \setminus EQ\left(\bold{O}\right)$, then for any path $P_{i,i'} \in \bold{P^S_{i,i'}}$ that does not include at least one node in $Nbr^{EQ\left(\bold{O}\right)}\left(j'\right)$, $f^{P_{i,i'}}_{i,i'} \leq v_{i'}\left(j',O_{j'}\right)$.

\end{enumerate}

\end{Lemma}

\begin{proof}

Consider all nodes in $A$ for which the maximum offer inducing path is of length $2$. Pick any of such nodes $i'$ and its neighbor $j$ along the maximum offer inducing path. Clearly $O_j \geq v_j\left(i', O_{i'}\right)$ from the construction of $O_j$. Assume $O_j > v_j\left(i', O_{i'}\right)$. Then $\exists \left(i'',j\right) \in E$ with $O_{i''} = v_{i''}\left(j, O_j\right) < v_{i''}\left(j, v_j\left(i', O_{i'}\right)\right)$, the inequality exists because the functions $v$ are strictly increasing. This implies that $O_{i''}$ is not the maximum path induced offer induced on $i''$ over all paths from $i$ to $i''$. Therefore by contradiction $O_j = v_j\left(i', O_{i'}\right)$ and $i'$ is connected to $i$ in $EQ\left(\bold{O}\right)$ through a maximum offer inducing path.

Now assume that all nodes  in $A$ for which the maximum offer inducing path is of length less than $n$ is connected to $i$ through a maximum offer inducing path. Then following lemma \ref{lem:max-path} for all nodes in $A$ for which the maximum offer inducing path is of length $n$ all nodes in $A$ along the path are connected to $i$ along the same path. Also by a similar argument as above all nodes in $A$ for which the maximum offer inducing path is of length $n$ is connected to $i$ through a maximum offer inducing path. Thus by induction, all nodes  in $A$ are connected to $i$ through a maximum offer inducing path. As a consequence all nodes in $B$ that belong to any of the maximum offer inducing paths are also connected to $i$ through the respective maximum offer inducing paths.

Now since all nodes in $B$ that do not belong to any maximum offer inducing paths are connected to at least one node in $A$, therefore they are also connected to $i$ through some path. Hence the equality subgraph is connected. The offer profile is stable because for all $j \in B$, $O_{j} = \max_{i' \in Nbr\left(j\right)}{v_j\left(i',O_{i'}\right)}$.

We now prove the second claim. Assume that there exists $\left(i',j\right) \in E \setminus EQ\left(\bold{O}\right)$ and path $P_{i,i'}$ from $i$ to $i'$ in $EQ\left(\bold{O}\right)$ that does not include any of the nodes in $Nbr^{EQ\left(\bold{O}\right)}\left(j\right)$. Then pick a node $i'' \in Nbr^{EQ\left(\bold{O}\right)}\left(j\right)$ and consider the path $P_{i,i'} \cup \{i',j,i''\}$. This path induces offers $v_j\left(i',O_{i'}\right)<O_j$ on $j$ and $v_{i''}\left(j,v_j\left(i',O_{i'}\right)\right)>v_{i''}\left(j,O_j\right)=O_{i''}$. This is a contradiction because by construction of $\bold{O}$, $O_{i''}$ was the maximum offer induced on $i''$ over all paths from $i$ to $i''$ given the offer $x$ on $i$. Hence the claims holds true.

We now prove the third claim. Assume there exists $\left(i',j'\right), \left(i',j''\right) \in E \setminus EQ\left(\bold{O}\right)$ and $v_{i'}\left(j',O_{j'}\right) > v_{i'}\left(j'',O_{j''}\right)$. From the second claim, all paths from $i$ to $i'$ in $EQ\left(\bold{O}\right)$ include a node from $Nbr^{EQ\left(\bold{O}\right)}\left(j'\right)$ and a node from $Nbr^{EQ\left(\bold{O}\right)}\left(j''\right)$. First we will show that either all paths from $i$ to $j'$ in $EQ\left(\bold{O}\right)$ include at least one node in $Nbr^{EQ\left(\bold{O}\right)}\left(j''\right)$ or all paths from $i$ to $j''$ in $EQ\left(\bold{O}\right)$ include at least one node in $Nbr^{EQ\left(\bold{O}\right)}\left(j'\right)$. Then we will show that it is actually the first case.
Assume there exists a path $P_{i,j'}$ from $i$ to $j'$ in $EQ\left(\bold{O}\right)$ not including any node in $Nbr^{EQ\left(\bold{O}\right)}\left(j''\right)$ and a path $P_{i,j''}$ from $i$ to $j''$ in $EQ\left(\bold{O}\right)$ not including any node in $Nbr^{EQ\left(\bold{O}\right)}\left(j'\right)$. Pick a path $P_{i, i'}$. This path includes a node $i''$ with the highest index among all nodes in $Nbr^{EQ\left(\bold{O}\right)}\left(j'\right) \cup Nbr^{EQ\left(\bold{O}\right)}\left(j''\right)$ with the highest index. Consider the subpath $P_{i'',i'} \subset P_{i,i'}$. If $i'' \in Nbr^{EQ\left(\bold{O}\right)}\left(j'\right)$, then the path $P_{i,j'} \cup\{j',i''\} \cup P_{i'',i'}$ is a path from $i$ to $i'$ not including any node in $Nbr^{EQ\left(\bold{O}\right)}\left(j''\right)$. If $i'' \in Nbr^{EQ\left(\bold{O}\right)}\left(j''\right)$, then the path $P_{i,j''} \cup\{j'',i''\} \cup P_{i'',i'}$ is a path from $i$ to $i'$ not including any node in $Nbr^{EQ\left(\bold{O}\right)}\left(j'\right)$. In either case, this contradicts the second claim, so either all paths from $i$ to $j'$ in $EQ\left(\bold{O}\right)$ include at least one node in $Nbr^{EQ\left(\bold{O}\right)}\left(j''\right)$ or all paths from $i$ to $j''$ in $EQ\left(\bold{O}\right)$ include at least one node in $Nbr^{EQ\left(\bold{O}\right)}\left(j'\right)$.
Now assume that all paths from $i$ to $j''$ in $EQ\left(\bold{O}\right)$ include at least one node in $Nbr^{EQ\left(\bold{O}\right)}\left(j'\right)$. Pick a path $P_{i,i'}$ from $i$ to $i'$ in $EQ\left(\bold{O}\right)$ and the node $i_*$ on the path $P_{i,i'}$ with the lowest
index among all nodes in $Nbr^{EQ\left(\bold{O}\right)}\left(j'\right) \cup Nbr^{EQ\left(\bold{O}\right)}\left(j''\right)$. From the assumption, $i_* \in Nbr^{EQ\left(\bold{O}\right)}\left(j'\right)$ 

Pick a node $i'' \in Nbr^{EQ\left(\bold{O}\right)}\left(j''\right)$ and consider the subpaths $P_{i,i_*} \subset P_{i,i'}$. The path $P_{i,i_*} \cup \{i_*,j',i'\} \cup \{i',j'',i''\}$ exists in $S$. The path $P_{i,i_*} \cup \{i_*,j',i'\} \cup \{i',j'',i''\}$ induces the following offers:\\
$v_{i'}\left(j',O_{j'}\right) > v_{i'}\left(j'',O_{j''}\right)$ on $i'$ \{from the assumption in the claim\}.\\
$v_{j''}\left(i',v_{i'}\left(j',O_{j'}\right)\right) < v_{j''}\left(i',v_{i'}\left(j'',O_{j''}\right)\right)=O_{j''}$ on $j''$.\\
$v_{i''}\left(j'',v_{j''}\left(i',v_{i'}\left(j',O_{j'}\right)\right)\right) > v_{i''}\left(j'',O_j''\right)=O_{i''}$ on ${i''}$.\\
This is a contradiction because by construction of $\bold{O}$, $O_{i''}$ was the maximum offer induced on $i''$ over all paths from $i$ to $i''$ given the offer $x$ on $i$. Hence, all paths from $i$ to $j'$ in $EQ\left(\bold{O}\right)$ include at least one node in $Nbr^{EQ\left(\bold{O}\right)}\left(j''\right)$.

We now prove the the fourth claim. Assume that there exists $\left(i',j'\right) \in E \setminus EQ\left(\bold{O}\right)$ and path $P_{i,i'}$ from $i$ to $i'$ in $S$ that does not include any of the nodes in $Nbr^{EQ\left(\bold{O}\right)}\left(j'\right)$. Then pick a node $i'' \in Nbr^{EQ\left(\bold{O}\right)}\left(j'\right)$ and consider the path $P_{i,i'} \cup \{i',j',i''\}$. Assume, $f^{P_{i,i'}}_{i,i'} > v_{i'}\left(j', O_{j'}\right)$. Then $v_{j'}\left(i',f^{P_{i,i'}}_{i,i'}\right) < O_{j'}$. The path $P_{i,i'} \cup \{i',j',i''\}$ induces offers $v_{j'}\left(i',f^{P_{i,i'}}_{i,i'}\right) < O_{j'}$ on $j''$ and $v_{i''}\left(j',v_{j'}\left(i',f^{P_{i,i'}}_{i,i'}\right)\right) > v_{i''}\left(j',O_{j'}\right) = O_{i''}$. This is a contradiction because by construction of $\bold{O}$, $O_{i''}$ was the maximum offer induced on $i''$ over all paths from $i$ to $i''$ given the offer $x$ on $i$. Hence the claims holds true.

\end{proof}


\subsection{Alternating Paths, Alternating Trees and Near-Perfect Matchings}

Given an offer profile $\bold{O}$ and a matching $M \subseteq EQ\left(\bold{O}\right)$, an {\bf alternating path} is a path within the equality subgraph $EQ\left(\bold{O}\right)$ with alternating pair of nodes share an edge in the matching and not in the matching $M$.

The matching also induces directionality on the alternating paths in the following way. Direct all edges $\left(i,j\right) \in M$ from $j$ to $i$ and all edges $\left(i,j\right) \notin M$ from $i$ to $j$.

An {\bf augmenting path} is an alternating path that starts and ends at an unmatched vertex.

An {\bf alternating tree} for a matching $M$ is a tree $T^M$ which contains exactly one unmatched node $r$ and has following properties:

\begin{itemize}

\item every node at odd distance from $r$ has degree 2 in the tree

\item all paths from $r$ are alternating paths

\item all leaf nodes are at even distance from $r$

\end{itemize}

Clearly, every alternating tree has one more node at even distance from $r$ than at odd distance from $r$.

A matching $M^{*}\left(\bold{O}\right) \subseteq EQ\left(\bold{O}\right)$ is a {\bf maximum matching} in the equality subgraph $EQ\left(\bold{O}\right)$ if $|M^{*}\left(\bold{O}\right)|=\max \{|M|:M \mbox{ is a matching in } EQ\left(\bold{O}\right)\}$.  A maximum matching $M^{*}\left(\bold{O}\right)$ can be obtained using the augmenting path algorithm \cite{Combinatorial-Optimization}. A matching $M$ is a maximum matching in $EQ\left(\bold{O}\right)$, if and only if there is no augmenting path in $EQ\left(\bold{O}\right)$ with respect to the matching $M$ \cite{Combinatorial-Optimization}.

A {\bf near-perfect matching} $M^{\circ}\left(\bold{O}\right) \subseteq EQ\left(\bold{O}\right)$ is a matching in the equality subgraph such that exactly one node is unmatched. A near-perfect matching exists only if $\||A|-|B|\|=1$.

Given a maximum matching $M^*\left(\bold{O}\right) \subseteq EQ\left(\bold{O}\right)$, an {\bf alternating forest} or a {\bf Hungarian forest} \cite{Paths-Trees-Flowers} $F^{M^*\left(\bold{O}\right)}$ is a collection of alternating trees rooted at nodes in $A$ induced by the matching. The number of alternating trees in a Hungarian forest is equal to the number of unmatched nodes in $A$. The Hungarian forest $F^{M^*\left(\bold{O}\right)}$ is the subgraph of $EQ\left(\bold{O}\right)$ induced by the set of nodes $X'$ reachable through alternating path from unmatched nodes in $A$.

Given a maximum matching $M^*\left(\bold{O}\right) \subseteq EQ\left(\bold{O}\right)$ and an alternating tree $T$, an {\bf expanding node} is a node $i' \in A^{T}$ with an edge with $j' \in B \setminus B^{T}$. We will refer $C^{T}$ to be the set of {\bf expanding nodes} for the tree $T$ and for each $i' \in C^{T}$, the respective {\bf expanding offer} $eo_{i'} = \max_{j' \in B \setminus B^{T}} {v_{i'}\left(j', O_{j'}\right)}$. We will also refer to $D^{T} = Nbr^E\left(C^{T}\right) \setminus B^{T}$ as the set of {\bf joining nodes} for the tree $T$.

An alternating tree $T^M$ for a matching $M$ is an {\bf alternating spanning tree} if it spans all the nodes in the network, i.e.- $X^{T^M} = A \cup B$.

An offer profile $\bold{O}$ is a {\bf stable alternating spanning tree generating offer profile} if $\bold{O}$ is stable and $EQ\left(\bold{O}\right)$ has a near perfect matching $M^{\circ}\left(\bold{O}\right)$ and associated alternating spanning tree $T^{M^{\circ}\left(\bold{O}\right)}$.


\begin{Lemma} \label{lem:neighbor-size}

Consider a set of nodes $A$ and $B$ with $|A|-|B|=1$ and a graph $S = \left(A \cup B, E\right)$ that has a near-perfect matching generating an alternating spanning tree. Then:

\begin{itemize}

\item Every $B' \subseteq B$ has least $|B'|+1$ neighboring nodes in $A$.

\item Every $A' \subset A$ has at least $|A'|$ neighboring nodes in $B$.

\end{itemize}

\end{Lemma}

\begin{proof}

Pick any subset $B' \subseteq B$. Since all nodes in $B$ are matched in a near-perfect matching, then from the Hall's theorem \cite{Halls-Theorem}, $B'$ has edges to at least $|B'|$ nodes in $A$. Since all nodes in $B'$ are interior nodes of an alternating tree, therefore, each node in $B'$ has one unique child it is matched to and one parent it is not matched to. Clearly, there is one parent node different from all the child nodes, or else, there will be a loop in the alternating tree. Hence, $B'$ has edges to at least $|B'|+1$ nodes in $A$ in the alternating tree within the network $S$.

For the second claim, pick any $A' \subset A$ and a near-perfect matching $M$ with an alternating tree $T$. If $A'$ does not contain the root of the alternating tree, then all nodes in $A'$ are matched and hence, the number of neighbors of the set $A'$ must be at least $|A|$. If $A'$ contains the root of the alternating tree, then pick a node $i'$ not in $A'$ and change the matching $M$ by switching the edges within the matching with the edges outside the matching along the alternating path from the root of $T$ to $i'$. This creates a new near-perfect matching and an alternating spanning tree whose root is at $i'$. For this matching, all the nodes in $A'$ are matched and hence the claim follows.

\end{proof}

\begin{proposition}

Consider a set of nodes $A$ and $B$ with $|A|-|B|=1$ and an equality subgraph $EQ\left(\bold{O}\right)$ that is connected. Then $EQ\left(\bold{O}\right)$ has a near-perfect matching if $\exists \mbox { a tree } T \subseteq EQ\left(\bold{O}\right)$ with the following properties:

\begin{enumerate}

\item $\forall j \in B$, $j$ is not a leaf node in $T$.

\item $\forall j \in B$, $j$ has exactly one child node in $T$.

\end{enumerate}

\end{proposition}

\begin{proof}

Clearly $\forall B' \subseteq B$, $| \cup_{j \in B'} Nbr(j)|>|B'|$. Hence following Hall's theorem, there exists a matching $M^\circ$ in $T$, such that all of $B$ is matched. Therefore, $M^\circ$ is a near-perfect matching.

\end{proof}



\section{Stable Alternating Spanning Tree Generating Offer Profiles}

In this section, we introduce three main lemmas about the existence, uniqueness and strict monotonicity of the stable alternating spanning tree generating offer profiles. Using this, we prove the main theorem of this section that helps extend the Hungarian algorithm to find the generalized stable matching. The main theorem introduces a set of continuous and strictly monotonic offer generating functions for each pair of nodes in the bipartite network.

\begin{Lemma} \label{lem:Stbl-Ineq}

Assume there exists two stable offer profiles $\bold{O}^1$ and $\bold{O}^2$.
\begin{itemize}
\item If $i \in A$ with $O^1_i \leq O^2_i$ and $(i,j) \in EQ\left(\bold{O}^2\right)$, then $O^1_j \geq O^2_j$.
\item If $i \in A$ with $O^1_i \geq O^2_i$ and $(i,j) \in EQ\left(\bold{O}^1\right)$, then $O^1_j \leq O^2_j$.
\item If $j \in B$ with $O^1_j \leq O^2_j$ and $(i,j) \in EQ\left(\bold{O}^2\right)$, then $O^1_i \geq O^2_i$.
\item If $j \in B$ with $O^1_j \geq O^2_j$ and $(i,j) \in EQ\left(\bold{O}^1\right)$, then $O^1_i \leq O^2_i$.
\end{itemize}

The resulting inequalities are strict when the conditioning inequalities are strict.

\end{Lemma}

\begin{proof}

We only need to prove the first statement and the rest follow similarly. To prove the first inequality, assume that  $O^1_{j} < O^2_{j}$. Then, $O^1_{j} < O^2_{j} = v_{j}\left(i, O^2_{i}\right) \leq v_{j}\left(i,O^1_{i}\right)$. The second inequality is due to the strict monotonicity of pareto payoff function. This implies that $\bold{O}^1$ is not stable contradicting our assumption. Hence, by contradiction, the first statement is true. 

\end{proof}

\begin{Lemma} \label{lem:Uniqueness-Offer}

Assume $|A|-|B|=1$. Pick any $i \in A$ and offer $o_i$. Assume there exists a stable alternating spanning tree generating offer profile $\bold{O}$ with $O_i=o_i$. Then $\bold{O}$ is the unique stable alternating spanning tree generating offer profile with $O_i=o_i$.

\end{Lemma}

\begin{proof}

Assume there exists two stable alternating spanning tree generating offer profiles $\bold{O}^1$ and $\bold{O}^2$ with $O^1_i=O^2_i=o_i$. Then $EQ\left(\bold{O}^1\right) \neq EQ\left(\bold{O}^2\right)$ or else $\bold{O}^1=\bold{O}^2$. Let $M^{\circ}\left(\bold{O}^1\right)$ and $M^{\circ}\left(\bold{O}^2\right)$ be the associated near perfect matchings and $T^{M^{\circ}\left(\bold{O}^1\right)}$ and $T^{M^{\circ}\left(\bold{O}^2\right)}$ be the associated alternating spanning trees. Clearly $T^{M^{\circ}\left(\bold{O}^1\right)} \neq T^{M^{\circ}\left(\bold{O}^2\right)}$ or else $\bold{O}^1=\bold{O}^2$. Without loss of generality assume that $i$ is unmatched in both $M^{\circ}\left(\bold{O}^1\right)$ and $M^{\circ}\left(\bold{O}^2\right)$. In both alternating spanning trees, the nodes at even distances from $i$ belong to $A$ and the nodes at odd distances from $i$ belong to $B$. Define
\begin{align*}
&AM_i=\{i' \in A ~:~ O^1_{i'}=O^2_{i'}\} \mbox{ and } BM_i=\{j' \in B ~:~ O^1_{j'}=O^2_{j'}\}\\
&AM^1_i=\{i' \in A ~:~ O^1_{i'}>O^2_{i'}\} \mbox{ and } BM^1_i=\{j' \in B ~:~ O^1_{j'}>O^2_{j'}\}\\
&AM^2_i=\{i' \in A ~:~ O^1_{i'}<O^2_{i'}\} \mbox{ and } BM^2_i=\{j' \in B ~:~ O^1_{j'}<O^2_{j'}\}
\end{align*}
Since both offer profiles are stable, therefore following the lemma \ref{lem:Stbl-Ineq}:

\begin{enumerate}

\item In the alternating tree $T^{M^{\circ}\left(\bold{O}^1\right)}$, the parents and children of nodes in $AM^1_i$ must belong to $BM^2_i$
\item In the alternating tree $T^{M^{\circ}\left(\bold{O}^1\right)}$, the parents and children of nodes in $BM^1_i$ must belong to $AM^2_i$
\item In the alternating tree $T^{M^{\circ}\left(\bold{O}^2\right)}$, the parents and children of nodes in $AM^2_i$ must belong to $BM^1_i$
\item In the alternating tree $T^{M^{\circ}\left(\bold{O}^2\right)}$, the parents and children of nodes in $BM^2_i$ must belong to $AM^1_i$

\end{enumerate}

Since, the nodes at odd distance in the alternating trees have exactly one child, and all nodes have exactly one parent, therefore:

\begin{enumerate}

\item Using 1, we have $|BM^2_i| \geq |AM^1_i|$ and from 4, we have $|AM^1_i| \geq |BM^2_i|+1$. This gives a contradiction $|BM^2_i| \geq |AM^1_i| \geq |BM^2_i|+1 > |BM^2_i|$.
\item Using 2, we have $|AM^2_i| \geq |BM^1_i|+1$ and from 3, we have $|BM^1_i| \geq |AM^2_i|$. This gives a contradiction $|AM^2_i| \geq |BM^1_i|+1 > |BM^1_i| > |AM^2_i|$.

\end{enumerate}

Therefore all the sets $AM^1_i, AM^2_i, BM^1_i, BM^2_i$ are empty which implies that $\bold{O}^1=\bold{O}^2$.

\end{proof}

\begin{Lemma} \label{lem:offer-monotonicity}

Assume $|A|-|B|=1$. Pick any $i \in A$ and assume that for all $o_i \leq c_i$ there exists a stable alternating spanning tree generating offer profile $\bold{O}$ with $O_i=o_i$. Pick any two stable alternating spanning tree generating offer profiles $\bold{O}^1$ and $\bold{O}^2$ with $O^1_i \leq c_i$ and $O^2_i \leq c_i$. If $O^1_i<O^2_i$, then $\forall i' \in A$, $O^1_{i'}<O^2_{i'}$ and $\forall j' \in B$, $O^1_{j'}>O^2_{j'}$.

\end{Lemma}

\begin{proof}

Clearly, there is no $i' \in A$, with $O^1_{i'}=O^2_{i'}$, otherwise by lemma \ref{lem:Uniqueness-Offer} $o^1_i=o^2_i$. Define
\begin{align*}
&AM^{'}_i=\{i' \in A ~:~ O^1_{i'}>O^2_{i'}\} \mbox{ and } BM^{'}_i=\{j' \in B ~:~ O^1_{j'}<O^2_{j'}\}
\end{align*}
Since both offer profiles are stable, therefore following the lemma \ref{lem:Stbl-Ineq}:

\begin{enumerate}

\item In the alternating tree $T^{M^{\circ}\left(\bold{O}^1\right)}$, the parents and children of nodes in $AM^{'}_i$ must belong to $BM^{'}_i$.
\item In the alternating tree $T^{M^{\circ}\left(\bold{O}^2\right)}$, the parents and children of nodes in $BM^{'}_i$ must belong to $AM^{'}_i$.

\end{enumerate}

Since, the nodes at odd distance in the alternating trees have exactly one child, and all nodes have exactly one parent, therefore using 1, we have $|BM^{'}_i| \geq |AM^{'}_i|$ and from 2, we have $|AM^{'}_i| \geq |BM^{'}_i|+1$. This gives a contradiction $|BM^{'}_i| \geq |AM^{'}_i| \geq |BM^{'}_i|+1 > |BM^{'}_i|$. Therefore, we have $AM^{'}_i$ and $BM^{'}_i$ empty which implies that $\forall i' \in A$, $O^1_{i'}<O^2_{i'}$ and $\forall j' \in B$, $O^1_{j'}>O^2_{j'}$.

\end{proof}

\begin{Lemma} \label{lem:existence-stable-spanning-offer-profile}

Assume $|A|-|B|=1$ and assume there exists a stable alternating spanning tree generating offer profile $\bold{O}$. Then $\forall i \in A$, and $o_i \leq O_i$, there exists a stable alternating spanning tree generating offer profile $\bold{O^*}$ with $O^*_i=o_i$.

\end{Lemma}

\begin{proof}

First we note that since there exists a near-perfect matching with an alternating spanning tree in $S$, therefore by lemma \ref{lem:neighbor-size}, every $A' \subset A$ has at least $|A'|$ neighbors in $B$ and every $B' \subseteq B$ has at least $|B'|+1$ neighbors in $A$.

We will prove the lemma by induction. First for $|B|=1$, lets call $A=\{a_1,a_2\}$ and $B=\{b\}$. Assume there exists a stable alternating spannign tree generating offer profile $\bold{O}$ with the equality subgraph $EQ\left(\bold{O}\right)$, a near perfect matching $M^{\circ}\left(\bold{O}\right)$ and the associated alternating spanning tree $T^{M^{\circ}\left(\bold{O}\right)}$. Then pick $o_{a_1} < O_{a_1}$ and set $\bold{O}'$ as $O'_{a_1}=o_{a_1}$, $O'_{b}=v_b\left(a_1, o_{a_1}\right)$ and  $O'_{a_2}=v_{a_2}\left(b, O'_{b}\right)$. Then clearly, $T^{M^{\circ}\left(\bold{O'}\right)}=T^{M^{\circ}\left(\bold{O}\right)}$. Hence the lemma is true when $|B|=1$.

Now assume that the lemma is true for all $1 \leq |B| < n$. We will show that the lemma is true for $|B|=n$. Pick $i \in A$ and $o_i < O_i$.

\begin{itemize}

\item For all $i' \in A$ set $O'_{i'} = \max_{P \in \bold{P}^S_{i,i'}}{f^P_{i,i'}\left(o_i\right)}$.
\item For all $j \in B$, set $O'_{j} = \max_{i' \in Nbr^S\left(j\right)}{v_j\left(i',O'_{i'}\right)}$.

\end{itemize}

From lemma \ref{lem:conn-stbl} the equality subgraph $EQ\left(\bold{O}'\right)$ is connected and the offer profile $\bold{O}'$ is stable. Also all nodes $j \in B$ have at least one edge in the equality subgraph. Pick a maximum matching $M^*$ in $EQ\left(\bold{O}'\right)$ and the alternating forest $F^*=F^{M^*}$ with respect to the matching $M^*$. If $F^*$ has exactly one alternating tree that spans all nodes in $X$, then $\bold{O}^* = \bold{O}'$ is the desired offer profile and the lemma is true for $|B|=n$. Otherwise, we proceed as follows. We will denote:

\begin{itemize}

\item $A^{F^*} = A \cap X^{F^*}$, $B^{F^*} = B \cap X^{F^*}$: the set of nodes in the Hungarian forest that belong to $A$ and $B$ respectively.

\item $A^{F^*}_{out} = A \setminus A^{F^*}$, $B^{F^*}_{out} = B \setminus B^{F^*}$: the set of nodes outside the Hungarian forest that belong to $A$ and $B$ respectively.

\end{itemize}

We will create a sequence of offer profiles $\bold{O}^t$ all with $O^t_i=o_i$ and show that the sequences converges to a stable alternating spanning tree generating offer profile in finite number of steps.
Let $I = \{i_1, i_2, ... , i_m\} \subset A$ be the set of unmatched nodes in $A$. Then the Hungarian forest $F^*$ has $m$ alternating trees each rooted at one of the nodes in $I$. For node $i_k$ we will call the alternating tree rooted at $i_k$ as $T_{i_k}$. Then $T_{i_k}$ has one more node in $A$ than in $B$.

For t=0, set

\begin{itemize}

\item $\bold{O}^t=\bold{O'}$.

\item $EQ^t = EQ\left(\bold{O^t}\right)$.

\item $M^t = M^*$ and $F^t=F^*$.

\item $A^t = A^{F^t}$, $B^t = B^{F^t}$ and $A^t_{out} = A^{F^t}_{out}$, $B^t_{out} = B^{F^t}_{out}$.

\item $m^t = m$, $I^t = \{i^t_1, i^t_2, ... , i^t_{m^t}\} = I$ and $\forall i_k \in I^t, T^t_{i^t_k} = T_{i_k}$, $i^t = i^t_1$, $T^t=T^t_{i^t_1}$.

\end{itemize}


At any time $t$, pick $i^t$ and the alternating tree $T^t$. By lemma \ref{lem:neighbor-size} there exists an expanding node $i' \in C^{T^t}$ with an expanding offer $eo_{i'} = \max_{j' \in D^{T^t}} {v_{i'}\left(j', O^t_{j'}\right)}$. Since $|B^{T^t}| < n$, therefore by the assumption for induction, within the subgraph $S_{|X^{T^t}}$, there exists a stable alternating spanning tree generating offer profile $\bold{O}^{i'}_{|X^{T^t}}$ with $O^{i'}_{i'}=eo_{i'}$. Consider any two expanding nodes $i'$ and $i''$, their respective expanding offers $eo_{i'}$ and $eo_{i''}$ and their respective stable alternating spanning tree generating offer profiles, $\bold{O}^{i'}_{|X^{T^t}}$ and $\bold{O}^{i''}_{|X^{T^t}}$. By lemma \ref{lem:offer-monotonicity}, if $eo_{i'} > O^{i''}_{i'}$, then $eo_{i''}<O^{i'}_{i''}$. By repeated application of lemma \ref{lem:offer-monotonicity}, we find that there exists an expanding node $i^*$ with expanding offers $eo_{i^*}$ and stable alternating spanning tree generating offer profile, $\bold{O}^{i^*}_{|X^{T^t}}$ such that for all expanding nodes $i'$, $eo_{i'} \leq O^{i^*}_{i'}$. For $t+1$, we set:

\begin{itemize}

\item For all $i' \in X^{T^t}$, set $O^{t+1}_{i'}=O^{i^*}_{i'}$ and for $i' \notin X^{T^t}$, set $O^{t+1}_{i'}=O^{t}_{i'}$.

\item $EQ^{t+1} = EQ\left(\bold{O}^{t+1}\right)$.

\item Since, $\bold{O}^{t+1}_{|X^{T^t}}$ is a stable alternating spanning tree generating offer profile within the subgraph $S_{|X^{T^t}}$, therefore there is a unique near-perfect matching $M^*_{|X^{T^t}} \subseteq EQ_{|X^{T^t}}$ that leaves $i^t$ unmatched. First set $M^{t+1}=M^{t}_{|X^{T^t}_{out}} \cup M^*_{|X^{T^t}}$. If there exists an augmenting path from $i^t$ with respect to the matching $M^{t+1}$ within the equality subgraph $EQ^{t+1}$, then switch the edges within the matching and outside the matching along the augmenting path to create a new matching $M^{t+1}$. Clearly this is a maximum matching in $EQ^{t+1}$ because there does not exist any other augmenting paths in $EQ^{t+1}$. $F^{t+1}$ is the Hungarian forest induced by the matching $M^{t+1}$ in the equality subgraph $EQ^{t+1}$.

\item $A^{t+1} = A^{F^{t+1}}$, $B^{t+1} = B^{F^{t+1}}$ and $A^{t+1}_{out} = A^{F^{t+1}}_{out}$, $B^{t+1}_{out} = B^{F^{t+1}}_{out}$.

\item If the matching size changed, then set $m^{t+1}=m^t-1$, for all $k<m^t$,set $i^{t+1}_k=i^{t}_{k+1}$, $I^{t+1} = \{i^{t+1}_1, ... , i^{t+1}_{m^{t+1}}\}$. If the matching did not change set $m^{t+1}=m^t$, for all $k \leq m^t$,set $i^{t+1}_k=i^{t}_{k}$, $I^{t+1} = I^{t}$. $\forall i^{t+1}_k \in I^{t+1}$, $T^{t+1}_{i^{t+1}_k}$ is the new alternating tree rooted at $i^{t+1}_k$. Set $T^{t+1} = T^{t+1}_{i^{t+1}_1}$.

\end{itemize}

We now show that the Hungarian forest satisfies certain properties at all time $t$.

\begin{proposition} \label{prop:character-forest}

The following hold about the Hungarian forest at any time $t$:

\begin{itemize}

\item $|B^{F^{t}}_{out}| \geq |A^{F^{t}}_{out}|$.

\item There is no edge between a node in  $A^{F^{t}}$ and a node in  $B^{F^{t}}_{out}$ in the equality subgraph $EQ\left(\bold{O}^t\right)$, i.e.- $\left(A^{F^{t}} \times  B^{F^{t}}_{out}\right) \cap EQ\left(\bold{O}^t\right) = \phi$. In other words, all the neighbors of $B^{F^{t}}_{out}$ belong to $A^{F^{t}}_{out}$.

\item All edges in the matching $M^t$ belong to $A^{F^{t}} \times B^{F^{t}} \cup A^{F^{t}}_{out} \times B^{F^{t}}_{out}$.

\item \label{prop:one-more-unmatched} The number of alternating trees in the Hungarian forest is one more than the number of unmatched nodes in $B^{F^{t}}_{out}$,  i.e.- $m^t = |B^{F^{t}}_{out}|+1$.

\end{itemize}

\end{proposition}

\begin{proof}

Since all nodes in $A^{F^{t}}_{out}$ are matched to nodes in $B^{F^{t}}_{out}$, the first claim holds.

Pick a node$i' \in A^{F^{t}}$ and an alternating tree $T$ that contains $i'$. If there exists an edge $\left(i',j'\right) \in EQ\left(\bold{O}^t\right)$ between node $i'$ and another node $j' \in B$, then:

\begin{itemize}

\item if the edge is in matching $M^t$, then the alternating path from the root of the alternating tree $T$ to $i'$ includes $j'$ and hence $j' \notin B^{F^{t}}_{out}$.

\item if the edge is not in the matching $M^*$, then the alternating path from the root of the alternating tree $T$ to $i'$ can be extended to include $j'$. Hence, $j' \notin B^{F^{t}}_{out}$.

\end{itemize}

Hence, the second claim follows and the third claim follows from it.

We now prove the fourth claim. The number of unmatched nodes in $B^{F^{t}}_{out}$ is $|B^{F^{t}}_{out}|-|A^{F^{t}}_{out}|$. Since all nodes in $A^{F^{t}}_{out}$ and $B^{F^{t}}$ are matched and the number of unmatched nodes in $A$ is one more than the number of unmatched nodes in $B$, therefore the number of unmatched nodes in $A^{F^{t}}$ is $|B^{F^{t}}_{out}|-|A^{F^{t}}_{out}|+1$. Since each unmatched node in $A^{F^{t}}$ is the root of a unique alternating tree, and each alternating tree has a unique root that belong to $A^{F^t}$ therefore the number of alternating trees in the Hungarian forest is one more than the number of unmatched nodes in $B^{F^{t}}_{out}$.



\end{proof}

We also observe that the offer profile $\bold{O}^t$ is stable at any time $t$ as shown in the following proposition.

\begin{proposition} \label{prop:stbl-offer-profile}

At any time $t$, the offer profile $\bold{O}^t$ is stable.

\end{proposition}

\begin{proof}

From \ref{lem:conn-stbl}, we know that the offer profile is stable at $t=0$.

Assume that for some $t \geq 0$, the offer profile $\bold{O}^t$ is stable. At iteration $t+1$, if the offers change, then:

From the construction of $\bold{O}^{t+1}$, and by lemma \ref{lem:offer-monotonicity} we know that:

\begin{align*}
&\forall i' \in A \setminus A^{T^t}, O^{t+1}_{i'}=O^{t}_{i'} \mbox{ and } \forall j' \in B \setminus B^{T^t}, O^{t+1}_{j'}=O^{t}_{j'}\\
&\forall i' \in A^{T^t}, O^{t+1}_{i'}<O^{t}_{i'} \mbox{ and } \forall j' \in B^{T^t}, O^{t+1}_{j'}>O^{t}_{j'}
\end{align*}

The edges in $E$ can be divided into four mutually exclusive subsets:

\begin{itemize}

\item $ E1 = E \cap \left(A^{T^t} \times B^{T^t}\right)$
\item $ E2 = E \cap \left(\left(A \setminus A^{T^t}\right) \times B^{T^t}\right)$
\item $ E3 = E \cap \left(A^{T^t} \times \left(B \setminus B^{T^t}\right)\right)$
\item $ E4 = E \cap \left(\left(A \setminus A^{T^t}\right) \times \left(B \setminus B^{T^t}\right)\right)$

\end{itemize}

From the construction, since $\bold{O}^{t+1}_{|X^{T^t}}$ is a stable alternating spanning tree generating offer profile within the subgraph $S_{|X^{T^t}}$, therefore,
\begin{align*}
&\forall \left(i',j'\right) \in E1, O_{i'}^{t+1} \geq v_{i'}\left(j',O_{j'}^{t+1}\right).
\end{align*}

Since $\bold{O}^t$ was stable therefore
\begin{align*}
&\forall \left(i',j'\right) \in E4, O_{i'}^{t+1} = O_{i'}^{t} \geq v_{i'}\left(j',O_{j'}^{t}\right) = v_{i'}\left(j',O_{j'}^{t+1}\right).
\end{align*}

From the construction  of $\bold{O}^{t+1}$ and the stability of $\bold{O}^t$ and since the pareto payoff functions are strictly decreasing therefore
\begin{align*}
&\forall \left(i',j'\right) \in E2, O_{i'}^{t+1} = O_{i'}^{t} \geq v_{i'}\left(j',O_{j'}^{t}\right) > v_{i'}\left(j',O_{j'}^{t+1}\right)
\end{align*}

Since for all $i' \in A_{t}$, $O^{t+1}_{i'} \geq eo_{i'}$, therefore
\begin{align*}
&\forall \left(i,j\right) \in E3, O_{i'}^{t+1} \geq eo_{i'} \geq v_{i'}\left(j',O_{j'}^{t}\right) = v_{i'}\left(j',O_{j'}^{t+1}\right)\\
\end{align*}

Therefore we see that
\begin{align*}
&\forall \left(i',j'\right) \in E, O_{i'}^{t+1} \geq v_{i'}\left(j',O_{j'}^{t+1}\right)
\end{align*}
and hence the offer profile $\bold{O}^{t+1}$ is stable.

By induction at any iteration $t \geq 0$, the offer profile $\bold{O}^t$ is stable.

\end{proof}

We also notice the following about the structural properties of the Hungarian forest and the offer profile at any time $t$ in the following proposition.

\begin{proposition} \label{prop:char-forest-offer}

At any time $t \geq 0$:

\begin{enumerate}

\item For all $t'>t$ and all $i' \in A$, $O^{t'}_{i'} \leq O^{t}_{i'}$ and for all $j' \in B$, $O^{t'}_{j'} \geq O^{t}_{j'}$.

\item For all unmatched nodes $j' \in B$, $O^t_{j'}=O^0_{j'}$.

\item If at any time $t$ a node $j \in B \setminus B^{T^t}$ has $O^t_{j} > O^t_{j}$, then $j$ has an alternating path for the matching $M^t$ in $EQ^t$ from some $j' \in B$ such that $O^t_{j'}=O^0_{j'}$.

\item If at any time $t$ a node $j \in B^{t}$ has $O^t_{j} > O^0_{j}$, then $j$ has a path in $EQ^t$ to some $j' \in B$ such that $O^t_{j'}=O^0_{j'}$.

\end{enumerate}

\end{proposition}

\begin{proof}

Since at any time $t$, for any expanding node $i' \in C^{T^t}$, the expanding offer $eo_{i'} < O^t_{i'}$, therefore by lemma \ref{lem:offer-monotonicity} for all $i' \in A^{T^t}$, $O^t_{i'}<O^{t-1}_{i'}$ and for all $i' \in A \setminus A^{T^t}$, $O^t_{i'}=O^{t-1}_{i'}$. Also, for all $j' \in B^{T^t}$, $O^t_{j'}>O^{t-1}_{j'}$ and for all $j' \in B \setminus B^{T^t}$, $O^t_{j'}=O^{t-1}_{j'}$. Hence the first claim holds.

Clearly, if the node $j' \in B$ is unmatched, it must not have been any alternating tree until time $t$. Hence $O^t_{j'}=O^0_{j'}$.

We now prove the third claim. first we note that $B \setminus B^{T^t} = \left(B^t \setminus B^{T^t}\right) \cup \left(B \setminus B^{t}\right)$.\\
If $j \in B^t \setminus B^{T^t}$, i.e.- $j$ is in the Hungarian forest but not in the alternating tree $T^t$ at time $t$, then $O^t_j=O^0_j$. To see this, assume that $j \in T^t_{i^t_k}$ for some $1<k<m^t$. Then all nodes $i' \in A^{T^t_{i^t_k}}$ have $O^t_{i'}=O^0_{i'}$ and therefore by stability of $\bold{O}^0$ and claim 1, $O^t_j=O^0_j$.\\
If a node $j \in B \setminus B^{t}$, i.e.- $j$ is not in the hungarian forest, then either $O^t_j=O^0_j$ or at some time $t' < t$ it was in the alternating tree $T^{t'}$. Pick $t^\circ$ to the be maximum of such times $t'$. Since, $j'$ was never in the hungarian forest after $t^\circ$, therefore at $t^\circ+1$, the alternating tree $T^{t^\circ}$ connected to a joining node $j'$ with an alternating path to an unmatched node $j^\circ$. From claim 2, $O^{t^\circ}_{j^\circ}=O^0_{j^\circ}$. At that time $t^\circ+1$, an alternating path was created from $j^\circ$ to $j$.  Clearly, $j^\circ$ was never in the hungarian forest until $t$, otherwise $j$ would be in the hungarian forest after time $t^\circ$ too. Therefore, $O^t_{j^\circ}=O^{t^\circ}_{j^\circ}=O^0_{j^\circ}$ and $j$ has an alternating path from $j^\circ$ for which $O^t_{j^\circ}=O^0_{j^\circ}$. Hence claim 3 holds.

We will prove the fourth claim by induction. Clearly the claim holds for $t=0$. Assume the claim holds for some $t \geq 0$. At time $t+1$, the alternating tree connects to one of the joining nodes $j \in B \setminus B^{T^t}$. By claim 3 $j$ has a path to some $j'$ with $O^{t+1}_{j'} = O^t_{j'} = O^0_{j'}$ and hence all nodes in the alternating tree have a path to a nodes $j'$ for which $O^{t+1}_{j'}=O^0_{j'}$. Hence by induction the claim holds.

\end{proof}

We now observe cetain properties of the equality subgraph outside the Hungarian forest at all times in the following proposition.

\begin{proposition} \label{prop:strct-outside-forest}

At any time $t \geq 0$, the following hold:

\begin{enumerate}

\item If there is more than one alternating tree in the Hungarian forest, then for all $j \in B$, such that there is an alternating path $P_{j,i} \subseteq EQ^t$ from $j$ to $i$, for all nodes $k' \in X^{P_{j,i}}$, $O^t_{k'}=O^0_{k'}$.

\item Assume that at time $t+1$ the alternating tree $T^{t}$ connects to a joining node $j \in D^{T^t}=Nbr^E\left(A^{T^t}\right) \setminus B^{T^t}$ such that there is an alternating path $P_{j,i} \subseteq EQ^t$ from $j$ to $i$,. Let $J \subset D^{T^t}$ be the set of joining nodes that the tree $T^t$ connects to at time $t+1$. Then each node in $D^{T^t}$ can be reached on alternating paths from nodes in $J$ in $EQ^{t}$.

\item If $i$ belongs to the alternating tree $T^t$ then $D^{T^t} = \phi$.

\item If there is more than one alternating tree in the Hungarian forest, the the node $i$ does not belong to the Hungarian forest.

\end{enumerate}

\end{proposition}

\begin{proof}

Clearly, the first claim hold at $t=0$.

For the fourth claim at $t=0$, assume that there is more than one alternating tree in the Hungarian forest and $i$ belongs to the Hungarian forest. Assume without loss of generality that $i$ belongs to an alternating tree $T$. Then $A^T \subset A$ has at least $|A^T|$ neighbors in $S$. Since $T$ has $|A^T|-1$ nodes in $B^T$, therefore at least one of these neighbors is outside $T$. Pick one such neighbor $j' \in B \setminus B^T$. By lemma \ref{lem:conn-stbl}, since the tree $T$ contains $i$, therefore, $T$ contains at least one neighbor $i'$ of $j'$ in the equality subgraph $EQ\left(\bold{O}^0\right)$. Then since $\left(i',j'\right) \in EQ\left(\bold{O}^0\right)$ and $i' \in A^T$, therefore by proposition \ref{prop:character-forest}, $j' \in B^T$. This contradicts our assumption that $j'$ is outside the alternating tree. Hence by contradiction, the fourth claim holds at $t=0$.

Assume the first and the fourth claims hold for some $t \geq 0$. Then at $t+1$ if for any node $k' \in X^P$, $O^{t+1}_{k'} \neq O^t_{k'} = O^0_{k'}$, then $k'$ must be in the alternating tree $B^{T^t}$ at $t$. This means that $i$ was in the alternating tree at $t$ which is a contradiction. Hence, the first claim is satisfied at $t+1$. Therefore by induction claim 1 holds true.

At $t+1$, assume that the alternating tree $T^t$ connects to a joining node $j$ that has an alternating path to $i$. Let $J \subset D^{T^t}$ be the set of joining nodes that the tree $T^t$ connects to at time $t+1$. Pick any $j' \in Nbr^E\left(A^{T^t}\right) \setminus B^{T^t}$ and assume that $j'$ is not reachable from an alternating path from any node in $J$ in $EQ^t$. Clearly, then $j'$ is not reachable from an alternating path from any node in $J$ in $EQ^{t+1}$. Pick $i' \in A^{T^t}$ such that $\left(i',j'\right) \in E$. By proposition \ref{prop:char-forest-offer}, $j'$ is reachable by an alternating path $P_{j^\circ, j'}$ from some $j^\circ$ with $O^t_{j^\circ}=O^0_{j^\circ}$ in both $EQ^{t}$ $EQ^{t+1}$. By stability of $\bold{O}^t$ and $\bold{O}^{t+1}$, and claim 1 of proposition \ref{prop:char-forest-offer}, all neighbors of $j^\circ$ in $EQ^0$ are also neighbors of $j$ in $EQ^{t}$ and $EQ^{t+1}$. Pick a neighbor $i^\circ$ of $j^\circ$ in $EQ^0$. From the assumption, no node in $X^{P_{j^\circ, j'}} \cup \{i^\circ\}$ belongs to the alternating tree $T^t$ and no node in $X^{P_{j^\circ, j'}} \cup \{i^\circ\}$ is reachable by an alternating path from any node in $J$. Consider the paths $P_{i,i'} \subset EQ^{t+1}$,$P_{j',j^\circ} \subset EQ^{t+1}$ and the path $P_{i,i^\circ} = P_{i,i'} \cup \{i', j'\} \cup P_{j',j^\circ} \cup \{j^\circ, i^\circ\}$ from $i$ to $i^\circ$. The offer induced on $i^\circ$ by the path $P_{i,i^\circ}$ for the offer $o_i$ on $i$ is

\begin{align*}
& f^{P_{i,i^\circ}}_{i,i^\circ}\left(o_i\right)\\
& = v_{i^\circ}\left(j^\circ,f^{P_{j',j^\circ}}_{j',j^\circ}\left(v_{j'}\left(i',f^{P_{i,i'}}_{i,i'}\left(o_i\right)\right)\right)\right)\\
& > v_{i^\circ}\left(j^\circ,f^{P_{j',j^\circ}}_{j',j^\circ}\left(O^{t+1}_{j'}\right)\right)  \mbox { since by stability of } \bold{O}^{t+1}, O^{t+1}_{j'} > v_{j'}\left(i', O^{t+1}_{i'}\right) = v_{j'}\left(f^{P_{i,i'}}_{i,i'}\left(o_i\right)\right)\\ 
& = v_{i^\circ}\left(j^\circ,O^{t+1}_{j^\circ}\right)\\
& = O^{t+1}_{i^\circ}\\
& = O^0_{i^\circ}
\end{align*}

This is a contradiction because by the definition of $\bold{O}^0$, $O^0_{i^\circ}$ is the maximum path induced offer on $i^\circ$ for the offer $o_i$ on $i$. Hence by contradiction, the second claim holds.

We now prove the third claim. Assume that at time $t$ $i$ belongs to the alternating tree $T^{t}$ and $D^{T^t} \neq \phi$. Pick $j' \in D^{T^t}$.
By proposition \ref{prop:char-forest-offer}, $j'$ is reachable by an alternating path $P_{j^\circ, j'}$ from some $j^\circ$ with $O^t_{j^\circ}=O^0_{j^\circ}$ in $EQ^{t}$. By stability of $\bold{O}^t$, and claim 1 of proposition \ref{prop:char-forest-offer}, all neighbors of $j^\circ$ in $EQ^0$ are also neighbors of $j$ in $EQ^{t}$. Pick a neighbor $i^\circ$ of $j^\circ$ in $EQ^0$. From the assumption, no node in $X^{P_{j^\circ, j'}} \cup \{i^\circ\}$ belongs to the alternating tree $T^t$. Consider the paths $P_{i,i'} \subset EQ^{t}$,$P_{j',j^\circ} \subset EQ^{t}$ and the path $P_{i,i^\circ} = P_{i,i'} \cup \{i', j'\} \cup P_{j',j^\circ} \cup \{j^\circ, i^\circ\}$ from $i$ to $i^\circ$. The offer induced on $i^\circ$ by the path $P_{i,i^\circ}$ for the offer $o_i$ on $i$ is

\begin{align*}
& f^{P_{i,i^\circ}}_{i,i^\circ}\left(o_i\right)\\
& = v_{i^\circ}\left(j^\circ,f^{P_{j',j^\circ}}_{j',j^\circ}\left(v_{j'}\left(f^{P_{i,i'}}_{i,i'}\left(o_i\right)\right)\right)\right)\\
& > v_{i^\circ}\left(j^\circ,f^{P_{j',j^\circ}}_{j',j^\circ}\left(O^{t}_{j'}\right)\right)  \mbox { since by stability of } \bold{O}^{t}, O^{t}_{j'} > v_{j'}\left(i', O^{t}_{i'}\right) = v_{j'}\left(f^{P_{i,i'}}_{i,i'}\left(o_i\right)\right)\\ 
& = v_{i^\circ}\left(j^\circ,O^{t}_{j^\circ}\right)\\
& = O^{t}_{i^\circ}\\
& = O^0_{i^\circ}
\end{align*}

This is a contradiction because by the definition of $\bold{O}^0$, $O^0_{i^\circ}$ is the maximum path induced offer on $i^\circ$ for the offer $o_i$ on $i$. Hence by contradiction, either $i$ does not belong to the alternating tree $T^t$ or $D^{T^t} = \phi$. This proves claim 3.

The fourth claim follows from the first claim as follows. Assume $i$ does not belong to the Hungarian forest until time $t-1$. Assume $i$ belongs to the Hungarian forest at time $t$ which implies that at time $t$, $T^{t-1}$ connected to a node $j$ with an alternating path to $i$ and hence $i$ belongs to the alternating tree $T^t$. Since there are more than one alternating trees in the Hungarian forest, then $|A^{T^t}|<|A|$, and therefore there is at least one joining node $j' \in D^{T^t}$ that $T^t$ can connect to by some expanding node $i' \in A^{T^{t}}$. This contradicts claim 3 that $D^{T^t} = \phi$. Hence $i$ does not belong to the Hungarian forest at time $t$. By induction, claim 4 holds.

\end{proof}


\begin{proposition} \label{prop:spanning-tree}

At any time $t$, if the Hungarian forest has only one alternating tree $T^t$, and this alternating tree contains $i$, then this alternating tree spans all nodes in $X$.

\end{proposition}

\begin{proof}

Since the Hungarian forest has only one alternating tree $T^t$ and it contains $i$, then by proposition \ref{prop:strct-outside-forest}, $D^{T^t} = \phi$. If $A^{T^t} \neq A$ then $|Nbr^E\left(A^{T^t}\right)| = |B^{T^t}|+|D^{T^t}| = |A^{T^t}|-1+|D^{T^t}|=|A^{T^t}|-1<|A^{T^t}|$ which contradicts lemma \ref{lem:neighbor-size}. Therefore $A^{T^t}=A$ and $|B^{T^t}|=|A^{T^t}|-1=|A|-1=|B|$. Therefore $B^{T^t}=B$. Therefore $T^t$ spans all nodes in $X$.


\end{proof}


\begin{proposition} \label{prop:matching-increases}

If at some time $t$, $m^t > 1$, then at some finite time $t^\circ$, the matching increases by $1$.

\end{proposition}

\begin{proof}

The root of the alternating tree $T^t$ is $i^t_1$. At each time $t'>t$, until $i^t_1$ is matched, the alternating tree $T^{t'-1}$ has the root $i^{t'-1}_1=i^{t}_1$ and one of the following happens:

\begin{enumerate}

\item $T^{t'-1}$ connects to a joining node $j \in D^{T^{t'-1}}$ such that $j$ has an alternating path to an unmatched node $j^\circ$.

\item $T^{t'-1}$ connects to a joining node $j \in D^{T^{t'-1}}$ such that $j$ has an alternating path to $i$.

\item $T^{t'-1}$ connects to joining nodes $J \subseteq D^{T^{t'-1}}$ such that no nodes in $J$ have an alternating path to either an unmatched node or to $i$.

\end{enumerate}

In the third case, $B^{T^{t'-1}} \subset B^{T^{t'}}$. Thus the alternating tree rooted at $i^t_1$ increases. Since, $B$ is finite, therefore the third case happens only finitely many times. Therefore, at some finite time $t^\circ$, either of the first two cases happen.  If the first case happens at $t^\circ$, then there is an augmenting path from $i^t_1$ to some unmatched node $j^\circ$. Thus, by construction of matching $M^{t^\circ}$ at time $t^\circ$, $i^t_1$ is matched in $M^{t^\circ}$ and the matching increases by $1$. If the second case happens at time $t^\circ$, then it will contradict proposition \ref{prop:strct-outside-forest} unless the first case happens along with the second case at time $t^\circ$. Hence at time $t^\circ$, $i^t_1$ is matched in $M^{t^\circ}$ and the matching increases by $1$.






\end{proof}

\begin{proposition} \label{prop:convergence-spanning-tree-profile}

At some finite time $t^*$, the offer profile $\bold{O}^{t^*}$ is a stable alternating spanning tree generating offer profile.

\end{proposition}

\begin{proof}

By repeated application of proposition \ref{prop:matching-increases}, we see that at some finite $t^\circ$, there is only one alternating tree in the Hungarian forest. At any time $t' >t^\circ$, unless $T^{t'-1}$ spans all nodes in $X$, the alternating tree $T^{t'-1}$ has the root $i^{t^\circ}_1=i^{t'-1}_1$ and one of the following happens:

\begin{enumerate}

\item $T^{t'-1}$ connects to a joining node $j \in D^{T^{t'-1}}$ such that $j$ has an alternating path to $i$.

\item $T^{t'-1}$ connects to joining nodes $J \subseteq D^{T^{t'-1}}$ such that no nodes in $J$ have an alternating path to $i$.

\end{enumerate}

In the second case, $B^{T^{t'-1}} \subset B^{T^{t'}}$. Thus the alternating tree rooted at $i^{t^\circ}_1$ increases. Since, $B$ is finite, therefore the second case happens only finitely many times. Therefore, at some finite time $t^*$, the first case happens. By proposition \ref{prop:spanning-tree} at time $t^*$, the alternating tree $T^{t^*}$ spans all nodes in $X$. Also by proposition \ref{prop:stbl-offer-profile}, the offer profie $\bold{O}^{t^*}$ is stable. Therefore, at some finite time $t^*$, the offer profile $\bold{O}^{t^*}$ is a stable alternating spanning tree generating offer profile.

\end{proof}

By proposition \ref{prop:convergence-spanning-tree-profile}, in finite time the sequence $\bold{O}^t$ converges to a stable alternating spanning tree generating offer profile $\bold{O}^*$ for $|B| = n$. By induction, the lemma holds. This completes the proof of the lemma.

\end{proof}

We now present the main result of this section that uses the results we developed.

\begin{Theorem} \label{thm:generating-functions}

Assume $|A|-|B|=1$. Assume there exists a stable alternating spanning tree generating offer profile  $\bold{O}$ with $O_j \geq 0$, $\forall j \in B$. Pick $i,i' \in A$ and $j \in B$ and define the functions:

\begin{itemize}

\item $f^S_{i,i'} ~:~ \left(-\inf,O_i\right] \rightarrow \left(-\inf,O_{i'}\right]$ with $f^S_{i,i'}\left(x\right) = O'_{i'}$, such that $\bold{O'}$ is the stable alternating spanning tree generating offer profile with $O'_i=x$ and $O'_j \geq 0$, $\forall j \in B$.
\item $f^S_{i',i} ~:~ \left(-\inf,O_{i'}\right] \rightarrow \left(-\inf,O_i\right]$ with $f^S_{i',i}\left(x\right) = O'_{i}$, such that $\bold{O'}$ is the stable alternating spanning tree generating offer profile with $O'_i=x$ and $O'_j \geq 0$, $\forall j \in B$.
\item $f^S_{i,j} ~:~ \left(-\inf,O_{i}\right] \rightarrow \left[O_{j}, \inf\right)$ with $f^S_{i,j}\left(x\right) = O'_{j}$, such that $\bold{O'}$ is the stable alternating spanning tree generating offer profile with $O'_i=x$ and $O'_j \geq 0$, $\forall j \in B$.

\end{itemize}

The functions  $f^S_{i,i'}$ and $f^S_{i',i}$ are inverse of each other and are continuous and strictly increasing for each pair $i,i' \in A$. The function $f^S_{i,j}$ is continuous and strictly decreasing for each pair $i \in A ,j \in B$.

\end{Theorem}

\begin{proof}

From lemmas \ref{lem:existence-stable-spanning-offer-profile} and \ref{lem:Uniqueness-Offer}, we know that the functions $f^S_{i,i'}$ and $f^S_{i',i}$ are well defined and are inverse of each other. Also from lemma \ref{lem:offer-monotonicity}, we know the functions $f^S_{i,i'}$ and $f^S_{i',i}$ are strictly increasing within the intervals $\left(-\inf,O_i\right)$ and $\left(-\inf,O_{i'}\right)$ respectively.

Now consider an open interval $\left(x,y\right) \subseteq \left(-\inf,O_i\right)$ with $y \leq O_i$ and set $x' = f^S_{i,i'}\left(x\right)$ and $y' = f^S_{i,i'}\left(y\right)$. Then since $f^S_{i,i'}$ and $f^S_{i',i}$ are strictly increasing, $\forall z \in \left(x,y\right)$, $f^S_{i,i'}\left(z\right)=\left(f^{}_{i',i}\right)^{-1}\left(z\right) \in \left(x',y'\right)$ and $\forall z' \in \left(x',y'\right)$, $f^S_{i',i}\left(z'\right)=\left(f^{S}_{i,i'}\right)^{-1}\left(z'\right) \in \left(x,y\right)$. Since, $x,y$ were arbitrarily picked, therefore for all open intervals in $\left(-\inf,O_i\right]$, the inverse images $f^{-1}_{i',i}\left(\left(x,y\right)\right) = f^S_{i,i'}\left(\left(x,y\right)\right)$ are open intervals. Therefore, by definition, $f^S_{i',i}$ is continuous. By similar reasoning, $f^S_{i,i'}$ is continuous.

Since for any $x \in \left(-\inf,O_i\right)$, the offer profile $\bold{O'}$ is stable alternating spanning tree generating offer profile, therefore
$f^S_{i,j}\left(x\right) = \max_{i' \in Nbr\left(j\right)}v_j\left(i',O'_{i'}\right) = \max_{i' \in Nbr\left(\right)}v_j\left(i',f^S_{i,i'}\left(x\right)\right)$.
Since the functions $v$ are continuous and strictly decreasing and $f^S_{i,i'}$ is continuous and strictly increasing and maximum of continous strictly decreasing functions is continuous and strictly decreasing, therefore $f^S_{i,j}\left(x\right)$ is continuous and strictly decreasing in $x$.

\end{proof}


\section{Algorithm}

In this section, we present an algorithm to find a stable and feasible offer profile for a bipartite network $S$. Existence of a stable and feasible offer profile proves the existence of a stable weighted matching in $S$. The algorithm is described as follows.

\subsection{Initialization}

We first define an initial offer profile $\bold{O}^0$ and the initial matching as follows.


For t=0, set

\begin{enumerate}

\item $\forall j \in B$, set $O_j^0 = 0$ and $\forall i \in A$, set $O_i^0 = \max_{j \in Nbr\left(i\right)} v_i\left(j,0\right)$.

\item Set $EQ^0 = EQ\left(\bold O^0\right)$.

\item Set $M^0 = M^*\left(EQ^0\right)$ some maximum matching in $EQ^0$.

\item if there is an alternating path from an unmatched node $i \in A$ with $O_i>0$ to a matched node $i^{\circ} \in A$ with $O_i^{\circ} = 0$, then switch the alternating edges from within the matching $M^0$ to outside the matching $M^0$ and vice-versa along the alternating path from $i$ to $i^{\circ}$. This leaves the matching size unchanged. Repeat this process until no such alternating paths are present in $EQ^0$.



\item Set $I^0 = \{i^0_1, i^0_2, ... , i^0_{m^0}\}$ as the set of unmatched nodes in $A$ with positive offers. Set $m^0 = |I^0|$, and $\forall i^0_k \in I^0, set T^0_{i^0_k} = T_{i^0_k}$ as the alterating tree induced by the matching $M^0$ rooted at $i^0_k$. Set $T^0=T^0_{i^t_1}$ as the alternating tree under consideration.

\end{enumerate}

\subsection{Iteration}

We iteratively change the offer profile to create a sequence of offer profiles. At each time $t \geq 0$, we compute the new offer profile $\bold{O}^{t+1}$ as follows.

While $I^t \neq \phi$, pick $T^t$.

\begin{enumerate}

\item $\forall i \in C^{T^t}$, set the expanding offer
\begin{align*}
&eo_i = \max \{\max_{j \in D^{T^t}}v_i\left(j,O_j^t\right), 0\}\\
&\mbox{and select } i^{c} \in C^{T^t} : f^{S_{|X^{T^t}}}_{i^{c},i}\left(eo_{i^{c}}\right) \geq eo_i, \forall i \in C^{T^t}
\end{align*}
Theorem \ref{thm:generating-functions} implies that such a node exists in $C^{T^t}$.

\item $\forall i \in A_{T^t} \cup B^{T^t}$ set $O_i^{t+1} = f^{S_{|X^{T^t}}}_{i^{c},i}\left(eo_{i^{c}}\right)$ and \\
$\forall i \notin A_{T^t} \cup B_{T^t}$ set $O_i^{t+1} = O_i^t$.

\item Set $EQ^{t+1} = EQ\left(\bold O^{t+1}\right)$.

\item Since, $\bold{O}^{t+1}_{|X^{T^t}}$ is a stable alternating spanning tree generating offer profile within the subgraph $S_{|X^{T^t}}$, therefore there is a unique near-perfect matching $M^*_{|X^{T^t}} \subseteq EQ^{t+1}_{|X^{T^t}}$ that leaves $i^t_1$ unmatched. First set $M^{t+1}=M^{t}_{|X^{T^t}_{out}} \cup M^*_{|X^{T^t}}$. If there exists an augmenting path from $i^t$ with respect to the matching $M^{t+1}$ within the equality subgraph $EQ^{t+1}$, then switch the edges within the matching and outside the matching along the augmenting path to create a new matching $M^{t+1}$. This increases the size of the matchign by $1$. If there is no augmenting path but there is an alternating path from $i^t_1$ to a matched node $i^{\circ} \in A$ with $O^{t+1}_{i^{\circ}} = 0$, then switch the alternating edges from within the matching $M^{t+1}$ and outside the matching $M^{t+1}$ along the alternating path from $i^t_1$ to $i^{\circ}$. This leaves the matching size unchanged but $i^t_1$ is now matched. Clearly this is a maximum matching in $EQ^{t+1}$ because there does not exist any other augmenting paths in $EQ^{t+1}$.

\item If $i^t_1$ is matched in $M^{t+1}$, or $O^{t+1}_{i^t_1} = 0$ then set $m^{t+1}=m^t-1$, for all $k<m^t$,set $i^{t+1}_k=i^{t}_{k+1}$, $I^{t+1} = \{i^{t+1}_1, ... , i^{t+1}_{m^{t+1}}\}$. Otherwise set $m^{t+1}=m^t$, for all $k \leq m^t$,set $i^{t+1}_k=i^{t}_{k}$, $I^{t+1} = I^{t}$. $\forall i^{t+1}_k \in I^{t+1}$, $T^{t+1}_{i^{t+1}_k}$ is the new alternating tree rooted at $i^{t+1}_k$. Set $T^{t+1} = T^{t+1}_{i^{t+1}_1}$.

\end{enumerate}


\subsection{Convergence}

We now show that the algorithm converges in finitely many iterations and the offer profile at the point of convergence is feasible and stable. We need to show that for some finite $t^*$, the maximum matching $M^{t^*}$ has all nodes with positive offers are matched. We also need to show that $\bold{O}^{t^*}$ is stable.

\begin{proposition} \label{prop:stbl-offer}

At any iteration $t \geq 0$, the offer profile $\bold{O}^t$ is stable.

\end{proposition}

\begin{proof}

At $t=0$ we know the offer profile is stable because by construction $\forall i \in A$, set $O_i^0 = \max_{j \in Nbr\left(i\right)} v_i\left(j,O^0_j\right)$. Assume that for some $t \geq 0$, the offer profile $\bold{O}^t$ is stable. If at iteration $t+1$, the offers do not change, then offer profile $\bold{O}^{t+1}$ is stable. If at iteration $t+1$, the offers change, then:

From the construction of $\bold{O}^{t+1}$, and by theorem \ref{thm:generating-functions} we know that:

\begin{align*}
&\forall i' \in A \setminus A^{T^t}, O^{t+1}_{i'}=O^{t}_{i'} \mbox{ and } \forall j' \in B \setminus B^{T^t}, O^{t+1}_{j'}=O^{t}_{j'}\\
&\forall i' \in A^{T^t}, O^{t+1}_{i'}<O^{t}_{i'} \mbox{ and } \forall j' \in B^{T^t}, O^{t+1}_{j'}>O^{t}_{j'}
\end{align*}

The edges in $E$ can be divided into four mutually exclusive subsets:

\begin{itemize}

\item $ E1 = E \cap \left(A^{T^t} \times B^{T^t}\right)$
\item $ E2 = E \cap \left(\left(A \setminus A^{T^t}\right) \times B^{T^t}\right)$
\item $ E3 = E \cap \left(A^{T^t} \times \left(B \setminus B^{T^t}\right)\right)$
\item $ E4 = E \cap \left(\left(A \setminus A^{T^t}\right) \times \left(B \setminus B^{T^t}\right)\right)$

\end{itemize}

From the construction, since $\bold{O}^{t+1}_{|X^{T^t}}$ is a stable alternating spanning tree generating offer profile within the subgraph $S_{|X^{T^t}}$, therefore,
\begin{align*}
&\forall \left(i',j'\right) \in E1, O_{i'}^{t+1} \geq v_{i'}\left(j',O_{j'}^{t+1}\right).
\end{align*}

Since $\bold{O}^t$ was stable therefore
\begin{align*}
&\forall \left(i',j'\right) \in E4, O_{i'}^{t+1} = O_{i'}^{t} \geq v_{i'}\left(j',O_{j'}^{t}\right) = v_{i'}\left(j',O_{j'}^{t+1}\right).
\end{align*}

From the construction  of $\bold{O}^{t+1}$ and the stability of $\bold{O}^t$ and since the pareto payoff functions are strictly decreasing therefore
\begin{align*}
&\forall \left(i',j'\right) \in E2, O_{i'}^{t+1} = O_{i'}^{t} \geq v_{i'}\left(j',O_{j'}^{t}\right) > v_{i'}\left(j',O_{j'}^{t+1}\right)
\end{align*}

Since for all $i' \in A_{t}$, $O^{t+1}_{i'} \geq eo_{i'}$, therefore
\begin{align*}
&\forall \left(i,j\right) \in E3, O_{i'}^{t+1} \geq eo_{i'} \geq v_{i'}\left(j',O_{j'}^{t}\right) = v_{i'}\left(j',O_{j'}^{t+1}\right)\\
\end{align*}

Therefore we see that
\begin{align*}
&\forall \left(i',j'\right) \in E, O_{i'}^{t+1} \geq v_{i'}\left(j',O_{j'}^{t+1}\right)
\end{align*}
and hence the offer profile $\bold{O}^{t+1}$ is stable.

By induction at any iteration $t \geq 0$, the offer profile $\bold{O}^t$ is stable.

\end{proof}

\begin{proposition} \label{prop:positive-offer-matched}

At any iteration $t \geq 0$, if there exists $j \in B$ with $O^t_j > 0$, then $j$ is matched in $M^t$.

\end{proposition}

\begin{proof}

At $t=0$, since all nodes in $B$ have offers $0$, this holds true. Assume that at some iteration $t \geq 0$, if there exists $j \in B$ with $O^t_j > 0$, then $j$ is matched in $M^t$. If the offer profile does not change at time $t+1$, then nothing changes and the proposition holds at $t+1$. Otherwise, at time $t+1$, one of the following happens:

\begin{enumerate}

\item $T^{t}$ connects to a joining node $j \in D^{T^{t}}$ such that $j$ has an alternating path to an unmatched node $j^\circ$.

\item  $T^{t}$ connects to a joining node $j \in D^{T^{t}}$ such that $j$ has an alternating path to a matched node $i^\circ \in A \setminus A^{T^{t}}$ with $O^{t+1}_{i^{\circ}} = 0$.

\item $O^{t+1}_i = 0$ for some $i \in A^{T^{t}}$.

\item $T^{t}$ connects to joining nodes $J \subseteq D^{T^{t}}$ but the above cases do not happen.

\end{enumerate}

At $t+1$, all nodes in $B^{T^t}$ are matched to the nodes in $A^{T^t}$ by the selected near-perfect matching $M^*_{|X^{T^t}}$ at time $t+1$. From the assumption at time $t$, we know that if there exists $j \in B$ with $O^t_j > 0$, then $j$ is matched in $M^t$.  Therefore, any node $j \in B \setminus B^{T^t}$ with $O^t_j>0$ is matched to some node in $A \setminus A^{T^t}$ in $M^t$ since $EQ^t \cap \left(A \setminus A^{T^t} \times B \setminus B^{T^t}\right) = \phi$. For all nodes $j \in B \setminus B^{T^t}$, we know that $O^{t+1}_j = O^t_j$.

In the fourth case $M^{t+1} = M^*_{|X^{T^t}} \cup M^t_{|X \setminus X^{T^t}}$ and hence all nodes in $B$ that were matched at time $t$ are matched at time $t+1$. In the first case, before exchanging any edges, we see that as in the third case, all nodes in $B$ that were matched at time $t$ are matched at time $t+1$. By exchanging the edges within the matching $M^{t+1}$ with edges outside the matching in the augmenting path, the nodes in the augmenting path still stay matched and it does not change matching outside the augmenting path, so all nodes in $B$ that were matched at time $t$ are matched at time $t+1$. For the second and third case, by the same argument, all nodes in $B$ that were matched at time $t$ are matched at time $t+1$.
Thus we see that all nodes in $B$ that were matched in $M^t$ are matched in $M^{t+1}$. Therefore any node $j \in B$ with $O^t_j>0$ that was matched in $M^t$ is matched in $M^{t+1}$. Also all nodes $j \in B$ for which $O^{t+1}_j > O^t_j \geq 0$ belong to $B^{T^t}$ and hence are matched in $M^{t+1}$. Therefore any node $j \in B$ with $O^{t+1}_j>0$ is matched in $M^{t+1}$. Thus by induction, the proposition holds.


\end{proof}

\begin{Theorem} \label{thm:final-thm}

There exists a finite time $t^*$ for which $O^{t^*}$ is feasible and stable.

\end{Theorem}

\begin{proof}

We will first show that at some finite time $t^*$, $m^{t^*} = 0$, i.e.- all nodes in $A$ with positive offers are matched. Then we will show that at that time $t^*$, $\bold{O}^{t^*}$ is feasible and stable.

For any $t$, the root of the alternating tree $T^t$ is $i^t_1$. At each time $t'>t$, until $i^t_1$ is matched, the alternating tree $T^{t'}$ has the root $i^{t'}_1=i^{t}_1$ and one of the following happens:

\begin{enumerate}

\item $T^{t'-1}$ connects to a joining node $j \in D^{T^{t'-1}}$ such that $j$ has an alternating path to an unmatched node $j^\circ$ or a matched node $i^\circ \in A \setminus A^{T^{t'-1}}$ with $O^{t'}_{i^{\circ}} = 0$.

\item $O^{t'}_i = 0$ for some $i \in A^{T^{t'-1}}$.

\item $T^{t'-1}$ connects to joining nodes $J \subseteq D^{T^{t'-1}}$ but the above two cases do not happen.

\end{enumerate}

In the third case, $B^{T^{t'-1}} \subset B^{T^{t'}}$. Thus the alternating tree rooted at $i^t_1$ increases. Since, $B$ is finite, therefore the third case happens only finitely many times. Therefore, at some finite time $t^\circ$, either of the first two cases happen. If the first case happes at $t^\circ$, then there is an augmenting path from $i^t_1$ to $j^\circ$ or there is an alternating path from $i^t_1$ to some $i^\circ$ with $O^{t^{\circ}}_{i^{\circ}} = 0$. Thus, by construction of matching $M^{t^\circ}$ at time $t^\circ$, $i^t_1$ is matched in $M^{t^\circ}$ and $m^{t^\circ} = m^{t^\circ - 1}$. If the second case happes at $t^\circ$, then there is an alternating path from $i^t_1$ to some $i$ with $O^{t^{\circ}}_{i} = 0$. Thus, by construction of matching $M^{t^\circ}$ at time $t^\circ$, $i^t_1$ is matched in $M^{t^\circ}$ and $m^{t^\circ} = m^{t^\circ - 1}$. Hence at time $t^\circ$, $i^t_1$ is matched in $M^{t^\circ}$ and $m^{t^\circ} = m^{t^\circ - 1}$.

Since $m^0$ is finite, and $m^t$ decreases by $1$ in finitely many iterations when $m^t > 0$, therefore at some finite time $t^*$, $m^{t^*} = 0$.

We now show that at $t^*$, $O^{t^*}$ is feasible and stable. From proposition \ref{prop:stbl-offer}, it follows that $O^{t^*}$ is stable. Also from proposition \ref{prop:positive-offer-matched} we see that all nodes $j \in B$ with $O^{t^*}_j > 0$ are matched in $M^{t^*}$. Since $m^{t^*} = 0$, therefore all nodes $i \in A$ with $O^{t^*}_i > 0$ are matched in $M^{t^*}$. Consider the matching $M^{t^*}$ and a split $s^{M^{t^*}}$ as follows:

\begin{align*}
& \forall (i,j) \in M^{t^*}, \mbox{ set } s^{M^{t^*}}\left(i\right) = u^{-1}_i\left(j, O^{t^*}_i\right),\\
& s^{M^{t^*}}\left(j\right) = u^{-1}_j\left(i, O^{t^*}_j\right),\\
& \forall \mbox { unmatched } i \in A, s^{M^{t^*}}\left(i\right) = 0.\\
& \forall \mbox { unmatched } i \in A, s^{M^{t^*}}\left(i\right) = 0.
\end{align*}

The above is a well defined split because for all$ (i,j) \in M^{t^*}$,

\begin{align*}
&s^{M^{t^*}}\left(i\right)+s^{M^{t^*}}\left(j\right)\\
&= u^{-1}_i\left(j, O^{t^*}_i\right)+s^{M^{t^*}}\left(j\right)\\
&= u^{-1}_i\left(j, v_i\left(j, O^{t^*}_j\right)\right)+s^{M^{t^*}}\left(j\right)\\
&=  u^{-1}_i\left(j, v_i\left(j, u_j\left(i, s^{M^{t^*}}\left(j\right)\right)\right)\right)+s^{M^{t^*}}\left(j\right)\\
&=  u^{-1}_i\left(j, u_i\left(j, w\left(i,j\right)-s^{M^{t^*}}\left(j\right)\right)\right)+s^{M^{t^*}}\left(j\right)\\
&=  w\left(i,j\right)-s^{M^{t^*}}\left(j\right)+s^{M^{t^*}}\left(j\right)\\
&= w\left(i,j\right).
\end{align*}

Hence, we have a weighted matching $\left(M^{t^*}, s^{M^{t^*}}\right)$ with payoff profile $\bold{U}\left(M^{t^*}, s^{M^{t^*}}\right) = \bold{O}^{t^*}$. Therefore $\bold{O}^{t^*}$ is a feasible offer profile.

\end{proof}

From theorem \ref{thm:final-thm}, we have proved the existence of a feasible and stable offer profile and have also found a stable weighted matching $\left(M^{t^*}, s^{M^{t^*}}\right)$ in the bipartite network $S$.

\section{Conclusions}

In this paper we extended the stable matching problem in bipartite networks to the general scenario where nodes derive value from the part of the split as well as the node they are matched to. This problem appears in real life scenarios and has applications in several problems such as marriage and matching theories, group selection, bargaining in networks and exchanges in networks. We studied a very general case when the value is continuous and strictly increasing in the part of the split and proved the existence of a stable weighted matching. The key ingredient to the proof is the existence of strictly monotonic and continuous stable alternating spanning tree generating offer profiles that helps us exend the Hungarian method to the generalized case. The method of computing a stable alternating spanning tree generating offer profile is not very efficient. However, with additional structure on the value functions and correlations between value functions, more efficient methods can be employed and will be an interesting line of future work.


\bibliographystyle{plain}
\bibliography{Ankur-Mani-bibtex}

\end{document}